\documentclass[12pt]{amsart}

\usepackage{psfrag}
\usepackage{color}

\usepackage{graphicx,graphics}
\usepackage{fullpage,amssymb,amsfonts,amsmath,amstext,amsthm,amscd,enumerate,verbatim,tikz}
\usepackage[T1]{fontenc}
\usetikzlibrary{matrix,arrows}

 \usepackage{url}

\begin{document}

\newtheorem{theorem}{Theorem}[section]
\newtheorem{result}[theorem]{Result}
\newtheorem{fact}[theorem]{Fact}
\newtheorem{conjecture}[theorem]{Conjecture}
\newtheorem{lemma}[theorem]{Lemma}
\newtheorem{proposition}[theorem]{Proposition}
\newtheorem{remark}[theorem]{Remark}
\newtheorem{corollary}[theorem]{Corollary}
\newtheorem{facts}[theorem]{Facts}
\newtheorem{question}[theorem]{Question}
\newtheorem{props}[theorem]{Properties}

\theoremstyle{definition}
\newtheorem*{rem}{Remark}
\newtheorem*{rems}{Remarks}
\newtheorem*{definition}{Definition}

\newtheorem{ex}[theorem]{Example}

\newcommand{\notes} {\noindent \textbf{Notes.  }}
\newcommand{\note} {\noindent \textbf{Note.  }}
\newcommand{\defn} {\noindent \textbf{Definition.  }}
\newcommand{\defns} {\noindent \textbf{Definitions.  }}
\newcommand{\x}{{\bf x}}
\newcommand{\z}{{\bf z}}
\newcommand{\B}{{\bf b}}
\newcommand{\V}{{\bf v}}
\newcommand{\T}{\mathcal{T}}
\newcommand{\Z}{\mathbb{Z}}
\newcommand{\Hp}{\mathbb{H}}
\newcommand{\D}{\mathbb{D}}
\newcommand{\R}{\mathbb{R}}
\newcommand{\N}{\mathbb{N}}
\renewcommand{\B}{\mathbb{B}}
\newcommand{\C}{\mathbb{C}}
\newcommand{\dt}{{\mathrm{det }\;}}
 \newcommand{\adj}{{\mathrm{adj}\;}}
 \newcommand{\0}{{\bf O}}
 \newcommand{\av}{\arrowvert}
 \newcommand{\zbar}{\overline{z}}
 \newcommand{\htt}{\widetilde{h}}
\newcommand{\ty}{\mathcal{T}}
\renewcommand\Re{\operatorname{Re}}
\renewcommand\Im{\operatorname{Im}}
\newcommand{\diam}{\operatorname{diam}}
\newcommand{\capac}{\operatorname{cap}}
\newcommand{\card}{\operatorname{card}}
\newcommand{\meas}{\operatorname{meas}}

\newcommand{\ds}{\displaystyle}
\numberwithin{equation}{section}

\newcommand{\eqn}{\begin{equation}}
\newcommand{\eqnend}{\end{equation}}

\renewcommand{\theenumi}{(\roman{enumi})}
\renewcommand{\labelenumi}{\theenumi}

\title{Superattracting fixed points of quasiregular mappings}

\author{Alastair Fletcher}
\address{Department of Mathematical Sciences, Northern Illinois University,
DeKalb, IL 60115-2888, USA}
\email{fletcher@math.niu.edu}
\author{Daniel A. Nicks}
\address{School of Mathematical Sciences, University of Nottingham, Nottingham NG7 2RD, UK}
\email{Dan.Nicks@nottingham.ac.uk}

\begin{abstract}
We investigate the rate of convergence of the iterates of an $n$-dimensional quasiregular mapping within the basin of attraction of a fixed point of high local index. A key tool is a refinement of a result that gives bounds on the distortion of the image of a small spherical shell. This result also has applications to the rate of growth of quasiregular mappings of polynomial type, and to the rate at which the iterates of such maps can escape to infinity.
\end{abstract}

\maketitle

\section{Introduction}

\subsection{Background}
One of the key features of holomorphic functions with respect to complex dynamics is their behaviour near fixed points. There is a complete classification of the possible iterative behaviours that can arise near a fixed point based on the value of the derivative at that point, see for example \cite{Milnor}. We say that a fixed point $z_0$ of a holomorphic function $f$ is \emph{superattracting} if $f'(z_0)=0$.
In this case, it follows from B\"ottcher's Theorem that there is a conjugation of $f$ to a map $z\mapsto z^d$ in a neighbourhood of $z_0$. That is, there exists a conformal map $B$ defined in a neighbourhood of $z_0$ such that $B(z_0) = 0$ and $B(f(z)) = B(z)^d$.
Hence $B(f^k(z)) = B(z)^{d^k}$ and from this it can be deduced that
\eqn  \log \log \frac{1}{|f^k(z)-z_0|} = k\log d + O(1), \quad \mbox{as } k\to\infty. \label{holorate1}\eqnend
It follows that if $z$ and $w$ are two points near to $z_0$, then there is a constant $\alpha$ such that
\eqn \frac{1}{\alpha} <  \frac{ \log |f^k(z)-z_0|}{\log |f^k(w)-z_0|} < \alpha  \label{holorate2}\eqnend
for all $k$.
Hence the rates at which different points are attracted towards the superattracting fixed point are comparable.

Quasiregular mappings are a natural generalization to higher dimensions of holomorphic functions in the plane. They exhibit many properties analogous to those of holomorphic functions and, in particular, there exist quasiregular versions of Picard's Theorem and Montel's Theorem.
See \cite{Rickman} for an introduction to the theory of quasiregular mappings. One aim of this article is to seek quasiregular analogies to the iterative results described above.

Holomorphic functions are differentiable everywhere and are locally injective precisely where the derivative is non-zero. In contrast, quasiregular mappings need only be differentiable almost everywhere.  To extend the notion of the multiplicity or valency of a holomorphic function, we define the \emph{local index} of a quasiregular mapping $f$ at the point $x$ to be
\[ i(x,f) = \inf_U \sup_y \card (f^{-1}(y)\cap U), \]
where the infimum is taken over all neighbourhoods $U$ of $x$. Thus $f$ is locally injective at $x$ if and only if $i(x,f)=1$.

The first research into iteration of quasiregular mappings studied the class of \emph{uniformly} quasiregular mappings; that is, those for which there is a uniform bound on the dilatation of the iterates. Behaviour near the fixed points of uniformly quasiregular mappings was studied in detail in \cite{HMM}. There it was shown that if a uniformly quasiregular mapping is locally injective at a fixed point, then it is in fact bi-Lipschitz there, whereas typically a quasiregular mapping is only locally H\"older continuous. A fixed point $x_0$ of a uniformly quasiregular mapping $f$ is called \emph{superattracting} if it is not locally injective at $x_0$. It follows from H\"older continuity (see Theorem \ref{rick2} below) that if $0\in E\subseteq\R^n$ and if $f:E\to \R^n$ is uniformly $K$-quasiregular with a superattracting fixed point $x_0=0$, then, for each $m\in\N$, there is a neighbourhood $U$ of $0$ and a constant $C$ such that, for all $x\in U$,
\[ |f^m(x)| \leq C|x|^{\mu_m},\]
where $\mu_m = (i(0,f)^m/K)^{1/(n-1)}$. Since $i(0,f)>1$, a large choice of $m$ will ensure that $\mu_m>1$ and hence that $f^{mj}\to 0$ as $j\to\infty$. Moreover, we can then deduce that the full sequence of iterates $f^k$ converges uniformly to $0$ on some neighbourhood of this fixed point.

For a general non-uniformly quasiregular mapping, the dilatation of the iterates will typically increase as the number of iterations increases. Even if a quasiregular mapping is not injective at some fixed point $x_0$, then it is possible that the iterates can fail to converge to $x_0$ on any neighbourhood of this fixed point. For example, the winding mapping given in polar coordinates in the plane by $(r,\theta) \mapsto (r, K\theta)$, for $K\in \N$, has both dilatation and local index at $0$ equal to $K$. It is easy to see that no other points near $0$ converge to $0$ under iteration.

We will see that the situation changes when the local index of a fixed point is greater than the inner dilatation. This condition appears to occur naturally in the iteration theory of polynomial type quasiregular maps; see \cite{BergFJ, FN, Sun00} and Section~\ref{sect:poly}.

\begin{definition}
Let $E\subseteq \R^n$ be a domain.
For a non-constant quasiregular mapping ${f:E\to \R^n}$, we will say that $x_0\in E$ is a \emph{strongly superattracting fixed point} if $f(x_0)=x_0$ and if ${i(x_0,f)>K_I(f)}$, where $K_I(f)$ denotes the inner dilatation of $f$. For such a point, we define the basin of attraction in $E$ as
\[ \mathcal{A}(x_0) = \{x\in E : f^k(x)\in E \mbox{ for all } k\in\N \mbox{ and } f^k(x)\to x_0\}. \]
\end{definition}

By the H\"{o}lder continuity of quasiregular maps (see Theorem~\ref{rick2} below), any strongly superattracting fixed point $x_0$ has a neighbourhood on which the iterates $f^k$ converge uniformly to $x_0$. Thus it follows that there is locally uniform convergence on the open set $\mathcal{A}(x_0)$.

\begin{rems}\mbox{}
\begin{enumerate}
\item For a holomorphic function $f$ defined on a domain in the complex plane, the inner dilatation $K_I(f)=1$ and so any superattracting fixed point is strongly superattracting in the above sense.
\item In \cite{FF}, B\"ottcher coordinates were constructed in a neighbourhood of infinity for degree two quasiregular mappings of the plane with constant complex dilatation. When the dilatation is less than $2$ in these examples, infinity is a strongly superattracting fixed point, and convergence to infinity is uniform on a neighbourhood of infinity.
\end{enumerate}
\end{rems}

\subsection{Statement of results}

 We refer to Section ~\ref{sect:defns} for the full definition of a quasiregular mapping and its associated inner dilatation $K_I(f)$ and outer dilatation $K_O(f)$. However, to state our results, we introduce here the following notation.
Let $f:E\to \R^n$ be quasiregular, $x\in E$ and $0\le r<d(x,\partial E)$.
Then we define
\eqn l(x,r) = \inf_{|y-x|=r} |f(y)-f(x)| \quad \ \mbox{ and } \ \quad L(x,r) = \sup_{|y-x|=r} |f(y)-f(x)|. \label{l,L defn} \eqnend
We may instead write $l(x,f,r)$ and $L(x,f,r)$ to make clear the dependence on the function.

Our main tool is the following result which gives bounds on the distortion of the image of a small spherical shell under a quasiregular map. It is a refinement of Lemma~3.9 of \cite{GMRV98}, though for completeness we shall give a full proof using path families.

\begin{theorem}\label{thm:ineqs}
Let $f:E\to \R^n$ be quasiregular and non-constant and let $x\in E$. Then there exist $C>1$ and $r_0>0$ such that, for all $0<T\le 1$ and all $r\in(0,r_0)$,
\eqn T^{-\mu} \le \frac{L(x,r)}{l(x,Tr)} \label{L/l} \le C^2 T^{-\nu} \eqnend
and
\eqn \frac{T^{-\mu}}{C^2} \le \frac{l(x,r)}{L(x,Tr)} \le T^{-\nu}, \label{l/L} \eqnend
where $\mu = (i(x,f)/K_I(f))^{1/(n-1)}$ and $\nu = (K_O(f)i(x,f))^{1/(n-1)}$. Moreover, $C$ depends only on $n$, $K_O(f)$ and $i(x,f)$.
\end{theorem}

Recall that when $x_0$ is a strongly superattracting fixed point of a quasiregular map $f$, the iterates $f^k$ converge locally uniformly to $x_0$ on $\mathcal{A}(x_0)$.
Our next result compares the rates at which different orbits can converge to such a strongly superattracting fixed point. We write $O^{-}(x):=\bigcup_{k\ge0}f^{-k}(x)$ for the backward orbit of a point.

\begin{theorem}\label{thm:superattr}
Let $f:E\to \R^n$ be quasiregular and let $x_0\in E$ be a strongly superattracting fixed point.
\begin{enumerate}
\item If $x,y\in\mathcal{A}(x_0)\setminus O^{-}(x_0)$, then there exists $N\in\N$ such that, for all large $k$,
\[  |f^{k+N}(y)-x_0| < |f^k(x)-x_0|. \]
\item Let $F$ be a compact subset of $\mathcal{A}(x_0)\setminus O^{-}(x_0)$. Then there exists $\alpha>1$ and $j\in\N$ such that
\[ \frac{1}{\alpha} < \frac{ \log |f^k(x)-x_0|}{\log | f^k(y)-x_0|} < \alpha\]
for all $x,y\in F$ and all $k\ge j$.
\end{enumerate}
\end{theorem}

The second part of Theorem~\ref{thm:superattr} generalizes \eqref{holorate2} from the holomorphic setting to the quasiregular setting. One may ask whether there is a similar generalization of \eqref{holorate1}. In this direction, if $x_0$ is a strongly superattracting fixed point of a quasiregular map $f$ and if $y$ is near to $x_0$, then iteration of the H\"older continuity inequality leads to
\[ k\log \mu +O(1) \le \log \log \frac{1}{|f^k(y)-x_0|} \le k\log \nu +O(1), \]
as $k\to\infty$, where $\mu$ and $\nu$ are as in Theorem~\ref{thm:ineqs} (cf.~\cite[Lemma~2.3]{F}).
The following example shows that this estimate is essentially the best possible. Take $1<K<2$ and real sequences $r_m\to 0$ and $s_m\to0$ such that $r_1=1$ and $0<r_{m+1}<s_m<r_m$. Now define a function $g$ on the unit disc in $\C$ by $g(0)=0$ and
\[ g(z) =
\begin{cases}
A_m z^2 |z|^{2K-2}, & s_m \le |z| < r_m, \\
B_m z^2 |z|^{2/K -2}, & r_{m+1} \le |z| < s_m,
\end{cases} \]
where the positive constants $A_m$ and $B_m$ are chosen to ensure continuity (with, say, $A_1=1$). This function $g$ is quasiregular with $K_I(g)=K_O(g)=K$ and $i(0,g)=2$. Thus $0$ is a strongly superattracting fixed point of $g$; also, $\mu=2/K$ and $\nu=2K$. By choosing the sequences $(r_m)$ and $(s_m)$ appropriately, each of the annuli $\{z: s_m\le|z|<r_m\}$ and $\{z: r_{m+1}\le|z|<s_m\}$ can be made to contain a large number of successive iterates $g^k(z)$. With $(r_m)$ and $(s_m)$ chosen suitably,  such a function can have
\[ \liminf_{k\to\infty} \frac1k \log \log \frac{1}{|g^k(z)|} = \log\frac2K = \log \mu \]
while also
\[ \limsup_{k\to\infty} \frac1k \log \log \frac{1}{|g^k(z)|} = \log 2K = \log \nu. \]
\begin{rem}
To summarize, in the holomorphic setting the iterates approach a superattracting fixed point at the rate specified by \eqref{holorate1} and hence different orbits are comparable as in \eqref{holorate2}. In the quasiregular case, Theorem~\ref{thm:superattr} gives an exactly analogous comparison for different orbits approaching a strongly superattracting fixed point. However, \emph{a priori} we have a less clear idea of the actual rate of approach.
\end{rem}

\subsection{Applications to polynomial type mappings}\label{sect:poly}

 A non-constant quasiregular mapping $f:\R^n\to\R^n$ is said to be of \emph{polynomial type} if $|f(x)|\to\infty$ as $|x|\to\infty$. If this is not the case, then $f$ has an essential singularity at infinity and is said to be of \emph{transcendental type}. It is well known that a map is of polynomial type if and only if
 \[ \deg f := \sup_{y\in\R^n}\operatorname{card} f^{-1}(y) < \infty. \]
 As usual, we denote the maximum modulus function by $M(r,f):=\sup_{|x|=r} |f(x)|$. Bergweiler \cite[Lemma~3.3]{B} has proved that if $S>1$ and $f$ is quasiregular of transcendental type, then
 \eqn \lim_{r\to\infty} \frac{M(Sr,f)}{M(r,f)} = \infty. \label{transM(Ar)} \eqnend
 He notes  that this implies that $\log M(r,f)/\log r \to\infty$ as $r\to\infty$ \cite[Lemma~3.4]{B} (see also~\cite{Jarvi}). If we now consider quasiregular maps of polynomial type, then H\"{o}lder continuity at infinity quickly gives that $\log M(r,f)/\log r$ is bounded above for large $r$.
 As before, this in turn implies that
 \[ \liminf_{r\to\infty} \frac{M(Sr,f)}{M(r,f)} < \infty. \]
The following improvement is an application of Theorem~\ref{thm:ineqs}.

\begin{theorem}\label{thm:M(Ar)}
Let $f:\R^n\to\R^n$ be quasiregular of polynomial type and let $S>1$. Then
\[ \limsup_{r\to\infty} \frac{M(Sr,f)}{M(r,f)} < \infty. \]
\end{theorem}

By combining \eqref{transM(Ar)} and Theorem~\ref{thm:M(Ar)}, we see that to characterize a map $f$ as being of polynomial type or of transcendental type it is enough to consider $M(Sr_j,f)/M(r_j,f)$ for some sequence $r_j\to\infty$ and some $S>1$.

The escaping set
\[ I(f) = \{x\in\R^n : f^k(x)\to\infty \mbox{ as } k\to\infty\} \]
of a polynomial type quasiregular mapping with $\deg f>K_I(f)$ was studied in \cite{FN} (cf.~\cite{BFLM} for transcendental type). For such mappings, the point at infinity can be viewed as a strongly superattracting fixed point and this allows us to obtain a result for the escaping set similar to Theorem~\ref{thm:superattr}. In fact, we find that all points of $I(f)$ escape to infinity at comparable rates, and so $I(f)$ coincides with the \emph{fast escaping set}
\[ A(f) = \{x\in\R^n: \exists L\in \N, |f^{k+L}(x)| > M^k(R,f) \text{ for all } k\in \N \}, \]
where $M^k(R,f)$ denotes the iterated maximum modulus (given, for example, by $M^2(R,f) = M(M(R,f),f)$) and where $R$ is chosen large enough so that $M^k(R,f) \to \infty$ as $k\to \infty$. Provided this last condition is satisfied, $A(f)$ is independent of the actual choice of $R$.
The fast escaping set consists of points which escape `as fast as possible', commensurate with the growth of the function. The fast escaping set for transcendental entire functions in the plane has received much recent attention, see for example \cite{RS},  and the fast escaping set for quasiregular mappings in $\R^n$ of transcendental type was investigated in \cite{BDF, BFN}.
For complex polynomials it is clear that $I(f)$ and $A(f)$ agree because every point which escapes does so at the same rate.

\begin{theorem}\label{thm:I(f)}
Let $f:\R^n\to\R^n$ be a quasiregular map of polynomial type that satisfies $\deg f>K_I(f)$.
\begin{enumerate}
\item There exist $S_0>1$ and $R>0$ such that, for all $S>S_0$, $r>R$ and $k\in\N$,
\[  SM^k(r,f) < m^k(Sr,f), \]
where $m(r,f):=\inf_{|x|=r} |f(x)|$ denotes the minimum modulus.
\item We have
\[ I(f) = A(f). \]
\item If $x, y \in I(f)$, then there exists $N\in\N$ such that, for all large $k$,
\[ |f^k(x)| < |f^{k+N}(y)|. \]
\item Let $F$ be a compact subset of $I(f)$. Then there exists $\alpha>1$ and $j\in\N$ such that
\[ \frac{1}{\alpha} < \frac{ \log |f^k(x)|}{ \log |f^k(y)| } < \alpha \]
for all $x,y \in F$ and all $k\ge j$.
\end{enumerate}
\end{theorem}

\begin{rem}
There are other alternative definitions for the fast escaping set, see \cite{BDF} for details. These all must coincide for those polynomial type mappings satisfying $\deg f>K_I(f)$ by Theorem~\ref{thm:I(f)}(ii), since  by definition each is a subset of $I(f)$ that contains the set $A(f)$ as given above.
\end{rem}

\subsection{Infinitesimal space}

Recall that quasiregular mappings need only be differentiable almost everywhere. As an extension of the notion of a tangent space, the infinitesimal space $T(x_0,f)$ was introduced in \cite{GMRV} to study points at which a quasiregular mapping fails to be well approximated by a non-degenerate linear map. It is defined as follows. Let $f:E \to \R^n$ be a non-constant $K$-quasiregular mapping and let $x_0 \in E$. For $x\in B(0,d(x_0,\partial E)/r)$, let
\begin{equation}\label{infsp1}
F_r(x) = \frac{ f(x_0 + rx) - f(x_0) }{\rho(x_0,f,r) },
\end{equation}
where $\rho(x_0,f,r)$ is the mean radius of the image of the ball $B(x_0,r)$ under the mapping $f$,
\[ \rho(x_0,f,r) = \left ( \frac{\meas f(B(x_0,r))}{\meas B(0,1)} \right )^{1/n}.\]
The infinitesimal space $T(x_0,f)$ consists of all locally uniformly convergent limit functions $h=\lim F_{r_j}$ as $r_j\to0$.
An element of $T(x_0,f)$ is called an infinitesimal mapping.

By results in \cite{GMRV}, $T(x_0,f)$ is always non-empty and every infinitesimal mapping is a $K$-quasiregular mapping of polynomial type that fixes~$0$, with local index at~$0$ and degree both equal to $i(x_0,f)$. If $f$ has a non-zero derivative at $x_0$, then $T(x_0,f)$ consists only of a normalized multiple of the linear mapping representing the derivative of $f$.

We use Theorem \ref{thm:ineqs} to prove that every mapping in an infinitesimal space has the same H\"older behaviour, compare with Theorem~\ref{rick2} below.

\begin{theorem}\label{thm:gmrv}
Using the notation above, let $h\in T(x_0,f)$. Then for $|x| \leq 1$ we have
\[ \frac{|x|^{\nu}}{C^2} \leq |h(x)| \leq C^2 |x|^{\mu},\]
and for $|x|\ge1$ we have
\[ \frac{|x|^{\mu}}{C^2} \leq |h(x)| \leq C^2 |x|^{\nu},\]
where $\nu = (K_O(f)i(x_0,f))^{1/(n-1)}$, $\mu = ( i(x_0,f)/K_I(f) )^{1/(n-1)}$ and the constant $C$ depends only on $n$, $K_O(f)$ and $i(x_0,f)$.
\end{theorem}

\section{Preliminaries}\label{sect:defns}

We denote by $B(x,r)$ the Euclidean ball centred at $x\in\R^n$ of radius $r>0$, by $S(x,r)$ the boundary of $B(x,r)$ and by $A(x,r,s)$ the spherical ring domain $\{y\in \R^n : r< |y-x| <s \}$. A continuous mapping $f:E \rightarrow \R^{n}$ defined on a domain $E \subseteq \R^{n}$ is called \emph{quasiregular} if it belongs to the Sobolev space $W^{1}_{n, \textrm{loc}}(E)$ and there exists $K \in [1, \infty)$ such that
\begin{equation}
\label{eq2.1}
\av f'(x) \av ^{n} \leq K J_{f}(x)
\end{equation}
almost everywhere in $E$. Here $J_{f}(x)$ denotes the Jacobian determinant of $f$ at $x \in E$. The smallest constant $K \geq 1$ for which (\ref{eq2.1}) holds is called the \emph{outer dilatation} $K_{O}(f)$. If $f$ is quasiregular, then there exists $K' \in[1, \infty)$ such that
\begin{equation}
\label{eq2.2}
J_{f}(x) \leq K' \inf _{\av h \av =1} \av f'(x) h \av ^{n}
\end{equation}
almost everywhere in $E$. The smallest constant $K' \geq 1$ for which (\ref{eq2.2}) holds is called the \emph{inner dilatation} $K_{I}(f)$. The \emph{dilatation} $K(f)$ of $f$ is the larger of $K_{O}(f)$ and $K_{I}(f)$, and we say that $f$ is $K$-quasiregular if $K(f)\le K$.

The following result shows that quasiregular mappings are locally H\"older continuous.

\begin{theorem}[{\cite[Theorem III.4.7]{Rickman}}]
\label{rick2}
Let $f:E\to \R^n$ be quasiregular and non-constant, and let $x\in E$. Then there exist positive numbers $\rho, A,B$ such that, for $y\in B(x,\rho)$,
\[ A|y-x|^{\nu} \leq |f(x)-f(y)| \leq B|y-x|^{\mu},\]
where $\nu = (K_O(f)i(x,f))^{1/(n-1)}$ and $\mu = (i(x,f)/K_I(f))^{1/(n-1)}$.
\end{theorem}

The next two lemmas concern the quantities $L(x,r)$ and $l(x,r)$ defined by \eqref{l,L defn}. The first of these is an  immediate consequence of \cite[Theorem~II.4.3]{Rickman}.

\begin{lemma}\label{lem:rick1}
Let $f:E\to\R^n$ be quasiregular and let $x\in E$. Then there exists $R_0>0$ such that if $r<R_0$ then
\[ L(x,r) \le Cl(x,r), \]
where $C$ depends only on $n$, $K_O(f)$ and $i(x,f)$.
\end{lemma}

\begin{lemma}\label{lem:incr}
Let $f:E \to \R^n$ be quasiregular and non-constant and let $x\in E$.  Then there exists $R>0$ such that $L(x,r)$ and $l(x,r)$ are increasing in $r$ for $0<r<R$.
\end{lemma}

\begin{proof}
Without loss of generality, we may assume that $x=0$.
Since non-constant quasiregular mappings are open mappings \cite[Theorem~I.4.1]{Rickman}, they satisfy a maximum principle and hence $\sup_{|y|\leq r}|f(y)-f(0)|$ is achieved on $S(0,r)$ and is equal to $L(0,r)$. Therefore $L(0,r)$ is increasing.

Non-constant quasiregular mappings are also discrete \cite[Theorem~I.4.1]{Rickman}, which means there exists $s>0$ such that $f(y)\ne f(0)$ for $0<|y|<s$. 
As $f(y)-f(0)$ is an open mapping, we see that $|f(y)-f(0)|$ has no local minima for $0<|y|<s$. We deduce that $l(0,r)$ has no local minima for $r\in(0,s)$.
Since $l(0,r)$ is continuous, this implies that $l(0,r)$ has at most one local maximum for $r\in(0,s)$. As $l(0,0)=0$ and $l(0,r)\ge0$, it follows that $l(0,r)$ is increasing over some subinterval $(0,R)\subseteq (0,s)$.
\end{proof}


The following notation will be used for certain path families in $\R^n$. If $E,F,D\subseteq\R^n$, then we denote by $\Delta(E,F;D)$ the family of paths that have one endpoint in each of $E,F$ and otherwise lie in $D$. For brevity, when $x\in\R^n$ and $0<r<s$, we write
\[ \Delta(x,r,s):= \Delta(S(x,r), S(x,s); B(x,s)). \]
See Chapter~II of \cite{Rickman} for details of the modulus $M(\Gamma)$ of a path family $\Gamma$. We note in particular that if every path in a family $\Gamma_2$ has a subpath in $\Gamma_1$, then $M(\Gamma_2)\le M(\Gamma_1)$, and also that
\eqn M(\Delta(x,r,s)) = \frac{\omega_{n-1}}{(\log (s/r))^{n-1}}, \label{M(annulus)} \eqnend
where $\omega_{n-1}$ is the $(n-1)$-dimensional surface area of the unit sphere in $\R^n$.

For a quasiregular mapping $f:E\to\R^n$, a domain $U$ compactly contained in $E$ is called a \emph{normal neighbourhood} of $x$ if $f(\partial U) = \partial f(U)$ and $U\cap f^{-1}(f(x))=\{x\}$. For such a normal neighbourhood, \cite[Proposition~I.4.10]{Rickman} gives that $N(f,U):=\sup_y\card f^{-1}(y)\cap U = i(x,f)$.

\section{Local behaviour of quasiregular mappings}

\begin{proof}[Proof of Theorem~\ref{thm:ineqs}]

Let $f$ and $x$ satisfy the hypotheses of the theorem.
Denote by $U(x,f,s)$ the component of $f^{-1}B(f(x),s)$ containing $x$. By \cite[Lemma~I.4.9]{Rickman} there exists $s_x>0$ such that $U(x,f,s)$ is a normal neighbourhood of $x$ and $f(U(x,f,s))=B(f(x),s)$ for $0<s\le s_x$. We choose $r_0$ small enough that $L(x,r)\le s_{x}$ for all $0<r<r_0$. We insist further that $r_0\le R_0$, where $R_0$ is as in Lemma~\ref{lem:rick1}.

We now take $r\in(0,r_0)$ and $0<T\le 1$ and begin by aiming to prove the left hand inequality in \eqref{L/l}.
We consider the sets
\[ U_l = U(x,f,l(x,Tr)), \qquad U_L=U(x,f,L(x,r))\]
and the path families
\[ \Gamma = \Delta(\partial U_l, \partial U_L ; U_L), \qquad \Gamma' = \Delta(f(x),l(x,Tr), L(x,r)).   \]
Note that by \eqref{M(annulus)},
\eqn
 M(\Gamma') = \omega_{n-1}\left(\log \frac{L(x,r)}{l(x,Tr)}\right)^{1-n}. \label{capf(X)}
\eqnend
Furthermore, as $\partial{U_l}\subseteq \overline{B(x,Tr)}$ and $B(x,r)\subseteq U_L$, every path in $\Gamma$ has a subpath in $\Delta(x,Tr,r)$ and thus
\eqn
 M(\Gamma) \le M(\Delta(x,Tr,r) )=\omega_{n-1}(\log 1/T)^{1-n}. \label{capX}
\eqnend

Note that $U_l$ and $U_L$ are both normal neighbourhoods of $x$. Hence $N(f,U_L)=i(x,f)$ and moreover $\Gamma$ is precisely the family of paths $\gamma$ in $U_L$ such that $f\circ\gamma\in\Gamma'$. Therefore V\"{a}is\"{a}l\"{a}'s inequality \cite[Corollary~II.9.2]{Rickman} states that
\[ M(\Gamma') \le \frac{K_I(f)}{i(x,f)}M(\Gamma). \]
Combining this with \eqref{capf(X)} and \eqref{capX} yields the left hand side of \eqref{L/l}. The left hand side of \eqref{l/L} follows immediately from the left hand side of \eqref{L/l} and Lemma~\ref{lem:rick1}.

We next prove the right hand inequality of \eqref{l/L}.
We may assume that $L(x,Tr)<l(x,r)$, for otherwise the result is clear. We consider the path family $\Delta:=\Delta(x,Tr,r)$. These paths lie in $B(x,r)$ and as this ball is contained in $U_L$ we have that $N(f,B(x,r))\le N(f,U_L)=i(x,f)$. Now the $K_O$-inequality \cite[Theorem II.2.4]{Rickman} states that
\eqn
 M(\Delta) \le K_O(f) i(x,f) M(f\Delta). \label{M(Delta)}
\eqnend
Let $\gamma$ be a path in $f\Delta$. Then $\gamma$ has one endpoint in $\{f(y): |y-x|=Tr\}$ and the other in $\{f(y): |y-x|=r\}$. It follows that $\gamma$ must have a subpath in $\Delta(f(x), L(x,Tr), l(x,r))$. Therefore
\[ M(f\Delta) \le M(\Delta(f(x), L(x,Tr), l(x,r)) = \frac{\omega_{n-1}}{(\log (l(x,r)/L(x,Tr)))^{n-1}}. \]
Combining this with \eqref{M(Delta)} and the fact that $M(\Delta)=\omega_{n-1}(\log 1/T)^{1-n}$ now gives the right hand side of \eqref{l/L}. The right hand inequality of \eqref{L/l} follows by again applying Lemma~\ref{lem:rick1}.
\end{proof}

\section{Iteration near strongly superattracting fixed points}

The following lemma deals with iterated versions of $L(x,r)$ and $l(x,r)$ from \eqref{l,L defn}. We set $l^1(x,r)=l(x,r)$ and $l^{k+1}(x,r)=l(x,l^k(x,r))$ for $k\ge1$, and define $L^k(x,r)$ similarly.

\begin{lemma}\label{lem:lk>Lk}
Let $f:E\to \R^n$ be quasiregular and let $x_0\in E$ be a strongly superattracting fixed point. Then there exist $T_0\in(0,1)$ and $r'>0$ such that, for all $0<T<T_0$, all $0<r<r'$ and all $k\in\N$,
\eqn
L^k(x_0,Tr) < Tl^k(x_0,r). \label{eqn:Lk>Lk}
\eqnend
\end{lemma}

\begin{proof}
Choose $T_0$ such that if $0<T<T_0$, then $T^{1-\mu}/C^2 > 1$, where $\mu$ and $C$ are as in Theorem~\ref{thm:ineqs} with $x=x_0$. This is possible because $\mu>1$ due to the fact that $x_0$ is a strongly superattracting fixed point.

Take $r_0$ as in Theorem~\ref{thm:ineqs} and $R_0$ as in Lemma~\ref{lem:incr}, each with $x=x_0$. We choose $0<r'\le\min \{ r_0,R_0\}$  small enough that $f$ maps $B(x_0,r')$ into itself, this being possible by Theorem~\ref{rick2}.

Now take $0<T<T_0$ and $0<r<r'$. Using the left hand side of \eqref{l/L} yields
\eqn
L(x_0,Tr) <  \frac{T^{1-\mu}}{C^2}L(x_0,Tr) \le Tl(x_0,r), \label{k=1}
\eqnend
which is the $k=1$ case of \eqref{eqn:Lk>Lk}. We next proceed by induction, assuming that \eqref{eqn:Lk>Lk} holds for some $k\in\N$. Then, using first Lemma~\ref{lem:incr} and this assumption, followed by \eqref{k=1}, we obtain that
\begin{align*}
L^{k+1}(x_0,Tr) = L(x_0, L^k(x_0,Tr)) &< L(x_0, Tl^k(x_0,r)) \\
&< Tl(x_0, l^k(x_0,r)) = Tl^{k+1}(x_0,r).
\end{align*}
Therefore  \eqref{eqn:Lk>Lk} holds for  all $k\in\N$ by induction.
\end{proof}

\begin{proof}[Proof of Theorem \ref{thm:superattr}]
Let $f$ satisfy the hypotheses of the theorem and without loss of generality assume that $x_0=0$. Choose $T_0$ and $r'$ as in the proof of Lemma~\ref{lem:lk>Lk} and let $0<T<T_0$.

By Lemma~\ref{lem:incr}, both $L(0,r)$ and $l(0,r)$ are increasing for $0<r<r'$, and hence it can be shown by induction that, for $x\in B(0,r')$,
\eqn
 l^k(0,|x|) \le |f^k(x)| \le L^k(0,|x|) \quad \mbox{ for all } k\in\N. \label{l<f<L}
\eqnend

We now prove parts (i) and (ii) of the theorem.

\begin{enumerate}
\item Given $x,y\in\mathcal{A}(0)\setminus O^{-}(0)$, there exist integers $k_0$ and $N$ such that
\[ |f^{k_0}(x)|<r' \quad \mbox{ and } \quad |f^{k_0+N}(y)|<T|f^{k_0}(x)|. \]
Using this together with \eqref{eqn:Lk>Lk}, \eqref{l<f<L} and the fact that $L(0,r)$ is increasing now shows that, for $k\ge k_0$,
\begin{align*}
|f^{k+N}(y)| \le L^{k-k_0}(0,|f^{k_0+N}(y)|) &< L^{k-k_0}(0,T|f^{k_0}(x)|) \\
&< Tl^{k-k_0}(0,|f^{k_0}(x)|) < |f^k(x)|.
\end{align*}

\item Let $F$ be a compact subset of $\mathcal{A}(0)\setminus O^{-}(0)$. Similarly to part~(i), since $f^k\to 0$ uniformly on $F$ we can find integers $k_0$ and $N$ such that
\[ \sup_{x\in F} |f^{k_0}(x)|<r' \quad \mbox{ and } \quad \sup_{y\in F}|f^{k_0+N}(y)|<\inf_{x\in F}T|f^{k_0}(x)|. \]
Here the infimum  is guaranteed to be positive because $F$ is a compact set disjoint from~$O^{-}(0)$. Now, for any $x,y\in F$, arguing exactly as in part~(i) gives that $|f^{k+N}(y)|  < |f^k(x)|$ for all $k\ge k_0$.

The function $f^N$ is quasiregular on $B(0,r')$ with $f^N(0)=0$. It follows from Theorem~\ref{rick2} that there exist $\alpha>1$ and $0<\rho<1$ such that, for $z\in B(0,\rho)$,
\[ |z|^\alpha < |f^N(z)|. \]
Note that there is a large integer $j$, say $j\ge k_0$, such that $|f^{k}(y)|<\rho$ for all $y\in F$ and all $k\ge j$. We may now conclude that, for any $x,y\in F$ and $k\ge j$,
\[ |f^k(y)|^\alpha < |f^{k+N}(y)| < |f^k(x)| < 1 \]
and thus
\[ \frac{ \log |f^k(x)|}{\log | f^k(y)|} < \alpha. \]
Finally, the required lower bound is established by interchanging $x$ and $y$.\qedhere
\end{enumerate}
\end{proof}

\section{Polynomial type mappings}

Let $f$ be a quasiregular mapping of polynomial type. Then $f$ can be extended to a continuous self-map of $\overline{\R^n}=\R\cup\{\infty\}$ by setting $f(\infty)=\infty$. Using appropriately modified definitions (see \cite[p.11]{Rickman}), this extended function $f$ is a quasiregular self-map of $\overline{\R^n}$ with local index at infinity $i(\infty,f)=\deg f$.

Let $g:\overline{\R^n}\to\overline{\R^n}$ be the inversion in the unit sphere given by $g(x)=x/|x|^2$, with $g(0)=\infty$ and $g(\infty)=0$. Note that this mapping satisfies $|g(x)|=1/|x|$ and $g=g^{-1}$. We shall consider the function $\tilde{f}=g\circ f\circ g$ conjugate to the polynomial type mapping $f$. In particular, after fixing $t>0$ large enough that $f(x)\ne0$ for $|x|>t$, it can be shown that $\tilde{f}:B(0,1/t)\to\R^n$ is quasiregular with $K_I(\tilde{f})\le K_I(f)$ and $i(0,\tilde{f})=i(\infty,f)=\deg f$. Observe that $\tilde{f}$ has a fixed point at $0$ and that this is strongly superattracting if $\deg f>K_I(f)$. Adopting the notation mentioned after \eqref{l,L defn}, we have that
\eqn
M(r,f) = \frac{1}{l(0,\tilde{f},\frac1r)} \quad \mbox{ and } \quad m(r,f) = \frac{1}{L(0,\tilde{f},\frac1r)}. \label{MlmL}
\eqnend

\begin{proof}[Proof of Theorem \ref{thm:M(Ar)}]
Take $\tilde{f}$ as above and $S>1$. Then \eqref{MlmL} and \eqref{L/l} of Theorem \ref{thm:ineqs} yield
\[ \frac{M(Sr,f)}{M(r,f)} \le \frac{M(Sr,f)}{m(r,f)} =  \frac{ L(0,\tilde{f},\frac1r) }{ l (0,\tilde{f},\frac{1}{Sr}) } \leq  C^2S^{\nu}\]
for all large $r$. Since the right hand side is independent of $r$, this proves the result.
\end{proof}

\begin{proof}[Proof of Theorem \ref{thm:I(f)}]
Let $f:\R^n \to \R^n$ be a quasiregular mapping of polynomial type such that the degree exceeds the inner dilatation. Then the map $\tilde{f}:B(0,1/t)\to\R^n$  described above is quasiregular and has $0$ as a strongly superattracting fixed point.
\begin{enumerate}
\item Applying Lemma~\ref{lem:lk>Lk} to $\tilde{f}$ gives $T_0\in(0,1)$ and $r'>0$ such that, for all $0<T<T_0$ and $0<r<r'$,
\[ L^k(0,\tilde{f},Tr) < Tl^k(0,\tilde{f},r) \quad \mbox{ for all } k\in\N. \]
We take $S_0=1/T_0$ and $R=1/r'$. Using \eqref{MlmL} now shows that, if $S>S_0$ and $r>R$, then
\[ m^k(Sr,f) = \frac{1}{L^k(0,\tilde{f},\frac{1}{Sr})} > \frac{S}{ l^k(0,\tilde{f},\frac{1}{r})} = S M^k(r,f) \]
for all $k\in\N$.

\item It is clear that $A(f)\subseteq I(f)$. Our aim is to prove the reverse inclusion.

Lemma~\ref{lem:incr} and \eqref{MlmL} together show that $m(r,f)$ is increasing for large $r$, and it follows that if $|x|>r$ then
\[ |f^k(x)| \ge m^k(r,f), \]
for all $k\in\N$.

We choose $S>S_0$ as in part~(i) and fix a large choice of $R$.
Now, given any $x\in I(f)$, there exists $L\in\N$ such that $|f^L(x)|>SR$. By the above and part~(i), we see that
\[ |f^{k+L}(x)| \ge m^k(SR,f) > SM^k(R,f) > M^k(R,f), \]
for all $k\in\N$. Therefore, $x\in A(f)$ as required.

\item Assume first that $x,y \in I(f)$ satisfy $|f^k(x)|>t$ and $|f^k(y)|>t$  for all $k\ge 0$ and  consider again the conjugate $\tilde{f} = g\circ f \circ g$. Then $\tilde{f}^k(g(x))$ and $\tilde{f}^k(g(y))$ are defined and non-zero for all $k$ and converge to $0$ as $k\to\infty$. Hence, $g(x)$ and $g(y)$ lie in the attracting basin $\mathcal{A}(0)$ with respect to $\tilde{f}$. Thus Theorem \ref{thm:superattr}(i) yields the existence of $N \in \N$ such that
\[ |\tilde{f}^{k+N}(g(y)) | < | \tilde{f}^k(g(x)) | \]
for all large $k$. The result follows since $|g(x) | = 1/|x|$.

For general $x,y\in I(f)$, we apply the above argument to $f^p(x)$ and $f^p(y)$ with a suitably large $p\in \N$.

\item The argument here is similar to that of part~(iii). Given a compact subset $F\subseteq I(f)$, we can choose a large integer $p$ such that all iterates $\tilde{f}^k$ are defined on $g(f^pF)$. We can then deduce the result by applying Theorem \ref{thm:superattr}(ii)  to the function $\tilde{f}$ and the compact set $g(f^pF)$. \qedhere
\end{enumerate}
\end{proof}

\section{Infinitesimal space}\label{sect:inf space}

\begin{proof}[Proof of Theorem \ref{thm:gmrv}]

Without loss of generality, we may assume that $x_0=0$. Let $r_0$ and $C$ be as given by Theorem~\ref{thm:ineqs} (applied to $f$ and $x=0$). Suppose that $h\in T(0,f)$ and that $h$ is realized as the limit of $F_{r_j}$ as $r_j\to 0$, where $F_{r_j}$ is as in \eqref{infsp1}.

It is clear that
\begin{equation}\label{lrl}
 l(x_0,f,r) \leq \rho(x_0,f,r) \leq L(x_0,f,r).
\end{equation}

Now take $x\in\R^n$ and note that, for all small $r_j$, the map $F_{r_j}$ is defined at $x$ and $r_j<r_0$. By \eqref{lrl} we have
\begin{equation} \label{infspeq1}
\frac{ l(0,f,r_j|x|)}{L(0,f,r_j)} \leq| F_{r_j}(x) | \leq \frac{ L(0,f,r_j |x| )}{l (0,f, r_j )} .
\end{equation}
Hence, if $|x|\leq1$, then Theorem \ref{thm:ineqs} implies that
\[ \frac{ |x|^{\nu } }{C^2} \leq |F_{r_j}(x) | \leq C^2 |x|^{\mu }.\]
On the other hand, if $|x| \ge1$, then we can rewrite \eqref{infspeq1} as
\[ \frac{ l(0,f,r_j')}{L(0,f,r_j'/|x|)} \leq| F_{r_j}(x) | \leq \frac{ L(0,f,r_j' )}{l (0,f, r_j'/|x| )}, \]
where $r_j' = r_j |x|$. Since $r_j'<r_0$ for all small $r_j$, we can use Theorem~\ref{thm:ineqs} again to obtain
\[ \frac{ |x|^{\mu} }{C^2} \leq |F_{r_j}(x) | \leq C^2|x|^{\nu}.\]
The proof is now complete, because $h$ is the limit of the mappings $F_{r_j}$ as $r_j\to0$.
\end{proof}

\end{document}